\documentclass[12pt]{amsart}
\usepackage[margin=1in]{geometry}
\usepackage{amsfonts}
\usepackage[utf8]{inputenc}
\usepackage[english]{babel}
\usepackage{amsthm}
\usepackage{cite}
\usepackage{amsmath}
\usepackage{xcolor}

\newtheorem{theorem}{Theorem}[section]
\newtheorem{lemma}[theorem]{Lemma}
\newtheorem{remark}[theorem]{Remark}
\newtheorem{example}[theorem]{Example}
\newtheorem{corollary}[theorem]{Corollary}
\newtheorem{proposition}[theorem]{Proposition}
\newtheorem{hypothesis}[theorem]{Hypothesis}
\newtheorem{definition}[theorem]{Definition}

\begin{document}
\title[Information Projection on Banach Spaces]{Information Projection on Banach spaces with Applications to State Independent KL-Weighted Optimal Control}
\author{Zachary Selk}
\address{Department of Mathematics, Purdue University, West Lafayette IN 47907, USA.}
\author{William Haskell}
\address{Krannert School of Management, Purdue University, West Lafayette IN 47907, USA.}
\author{Harsha Honnappa}
\address{School of Industrial Engineering, Purdue
University, West Lafayette IN 47907, USA.}
\email{[zselk,whaskell,honnappa]@purdue.edu}

\thanks{H. Honnappa is supported by the NSF grant DMS-1812197.}
  

\date{\today}
\maketitle
\begin{abstract}
This paper studies constrained information projections on Banach spaces with respect to a Gaussian reference measure. Specifically our interest lies in characterizing projections of the reference measure, with respect to the KL-divergence, onto sets of measures corresponding to changes in the mean (or {\it shift measures}). As our main result, we give a portmanteau theorem that characterizes the relationship among several different formulations of this problem. In the general setting of Gaussian measures on a Banach space, we show that this information projection problem is equivalent to minimization of a certain Onsager-Machlup (OM) function with respect to an associated stochastic process. We then construct several reformulations in the more specific setting of classical Wiener space. First, we show that KL-weighted optimization over shift measures can also be expressed in terms of an OM function for an associated stochastic process that we are able to characterize. Next, we show how to encode the feasible set of shift measures through an explicit functional constraint by constructing an appropriate penalty function. Finally, we express our information projection problem as a calculus of variations problem, which suggests a solution procedure via the Euler-Lagrange equation. We work out the details of these reformulations for several specific examples.

\end{abstract}

\section{Introduction}
This paper studies constrained information projections on Banach spaces with respect to a Gaussian reference measure. Specifically, our interest lies in characterizing projections of the reference measure onto sets of measures corresponding to changes in the mean (or {\it shift measures}). Information projection onto shift measures emerges in a number of domains of applied probability including stochastic optimal control, approximate inference, and large deviations analysis. Let $\mathcal B$ be a separable Banach space, $\mu_0$ a centered Borel Gaussian measure on $\mathcal B$, and $\mu^\ast\sim \mu_0$ another Borel measure where $\sim$ denotes equivalence of measures. The symbol $\mathcal H_{\mu_0} \hookrightarrow \mathcal B$ represents the Cameron-Martin space associated to $\mu_0$. By the Cameron-Martin Theorem~\cite[Theorem 3.41]{hairer2009introduction}, the shift measure $T_h^\ast(\mu_0)(\cdot):=\mu_0(\cdot - h)$ for $h\in \mathcal B$ is absolutely continuous with respect to $\mu_0$ if and only if $h\in \mathcal H_{\mu_0} $. Let $\mathcal P$ denote the set of all shift measures which are absolutely continuous with respect to $\mu_0$. We are interested in computing the information projection of $\mu^*$ onto $\mathcal{P}$, $\inf_{\mu\in \mathcal P} D_{KL}(\mu||\mu^\ast)$, where $D_{KL}(\mu||\mu^\ast)$ is the KL-divergence of $\mu$ with respect to $\mu^\ast$.

The main result of this paper characterizing the information projection is consolidated in the following {\it portmanteau} theorem. For this result, we let $\operatorname{val}(\cdot)$ and $\operatorname{sol}(\cdot)$ denote the optimal value and the set of optimal solutions of an optimization problem, respectively.

\begin{theorem}\label{theorem:Portmanteau}[A Portmanteau Theorem]
Let $\mathcal B$ be a separable Banach space. Let $\mu_0$ be a Gaussian measure on $\mathcal B$ (see Definition \ref{def:Gaussian-measure}) with Cameron-Martin space $\mathcal H_{\mu_0}$ (see Definition \ref{def:CM-space}). Let $C:\mathcal B\to \mathbb R$ be a functional satisfying Hypothesis \ref{hypothesis:FE}. Furthermore, assume that the functional $\Phi:\mathcal B\to \mathbb R$ defined by $\Phi(z)=E_{\mu_0}[C(\omega+z)]$ exists and satisfies Hypothesis \ref{hypothesis:OM}. Define the probability measure $\mu^\ast$ on $\mathcal B$ with density 
$$\frac{d\mu^\ast}{d\mu_0}:=\frac{e^{-C}}{E_{\mu_0}[e^{-C}]},$$
and the probability measure $\tilde \mu$ on $\mathcal B$ with density
$$\frac{d\tilde{\mu}}{d\mu_0}:=\frac{e^{-\Phi}}{E_{\mu_0}[e^{-\Phi}]}.$$
Let $\mathcal P$ be the set of Gaussian shift measures which are absolutely continuous with respect to $\mu_0$ (defined in Eq. \eqref{eq:P-set-def}). Consider the following three optimization problems:

(a) (Information projection) $$\mathbb K:=\inf_{\mu\in \mathcal P} D_{KL}(\mu||\mu^\ast).$$

(b) (State independent KL-weighted control) 
$$\mathbb P:=\inf_{\mu\in \mathcal P} \left\{E_\mu [C]+D_{KL}(\mu||\mu_0)\right\}.$$

(c) (Mode of $\tilde \mu$) $$\mathbb M:=\inf_{z\in \mathcal H_{\mu_0}} OM_\Phi(z).$$
Then, the optimal values $\operatorname{val}(\mathbb K)$, $\operatorname{val}(\mathbb P)$, and $\operatorname{val}(\mathbb M)$ are all attained (and so all three problems have optimal solutions). By the Cameron-Martin theorem (see Theorem \ref{theorem:CM}), we can identify each shift measure $\mu_z \in \mathcal P$ with its corresponding shift $z\in \mathcal H_{\mu_0}$, and so we have
\[
   \operatorname{sol}(\mathbb K)=\operatorname{sol}(\mathbb P) \equiv \operatorname{sol}(\mathbb M).
\]
In the above, we have used equivalence ``$\equiv$'' for shorthand to denote that $\mu_z\in \operatorname{sol}(\mathbb P)$ (or $\mu_z\in \operatorname{sol}(\mathbb K)$) if and only if its corresponding shift $z\in \operatorname{sol}(\mathbb M).$

Furthermore, suppose $(\mathcal B,\mu_0)$ is the classical Wiener space (see Definition \ref{def:Wiener-space}) and $C$ satisfies Hypothesis \ref{hypothesis:L2}. Let $L(t,z,\dot z)$ be the Lagrangian (defined in Eq. \eqref{eq:lagrangian-definition}), and suppose $L(t,z,\dot z)$ satisfies Hypothesis \ref{assum:coercivity}. Consider the following calculus of variations problem:

(d) (Calculus of variations)
$$
\mathbb L : \inf_{q \in W_0^{1,2}[0,T]} \left\{ \mathcal{L}(q,\dot{q};\,f) := \frac{1}{2} \int_0^T L(t,q(t),\dot{q}(t);\,f)dt \text{ s.t. } q(0) = 0 \right\}.
$$
Then, the optimal value $\operatorname{val}(\mathbb L)$ is attained (so Problem $\mathbb L$ has an optimal solution). We continue to identify each shift measure $\mu_z \in \mathcal P$ with its corresponding shift $z\in \mathcal H_{\mu_0}$ by the Cameron-Martin theorem (see Theorem \ref{theorem:CM}), and so we have

\[
     \operatorname{sol}(\mathbb K)=\operatorname{sol}(\mathbb P) \equiv \operatorname{sol}(\mathbb M)=\operatorname{sol}(\mathbb L).
\]
Again, as above we use $\equiv$ to denote that $\mu_z\in \operatorname{sol}(\mathbb P)$ (or $\mu_z\in \operatorname{sol}(\mathbb K)$) if and only if its corresponding shift $z\in \operatorname{sol}(\mathbb M)$ (or $z\in \operatorname{sol}(\mathbb L)$).

\end{theorem}

In summary, the principal technical contributions of the present paper are:

\noindent (i)  For Gaussian reference measures over an arbitrary Banach space, we show that the optimal Gaussian shift measure of the (constrained) information projection problem corresponds precisely to the minimizer of the Onsager-Machlup function of an associated Gibbs measure over the Banach space (that is fully characterized). As a direct consequence, the optimal open loop KL-weighted control function can be viewed as finding the most likely path of a stochastic process associated with the Gibbs measure. 

\noindent (ii) We further characterize the information projection onto Gaussian shift measures with respect to a reference Wiener measure. In particular, we derive a single functional constraint to express the feasible region $\mathcal P$. This functional constraint reveals that the feasible region of this optimization problem is non-convex, and we also find that standard convex inner and outer relaxations only yield ``trivial'' approximations.

\noindent (iii) Finally, we re-formulate the information projection problem and show that the optimal solution corresponds to a calculus of variations problem. More precisely, we show that the solution is completely characterized by the solution of an Euler-Lagrange equation. This yields a tractable representation, which we illustrate through a number of examples.

\paragraph{\bf Related Literature.} Information projections onto shift measures appear in a number of domains of applied probability, including stochastic optimal control, large deviations theory and approximate inference. Kullback-Liebler (KL) or relative-entropy weighted control~\cite{Bierkens,Todorova,Todorovb} studies a class of restricted control problems where the total cost of a given control policy can be expressed as the KL-divergence between the path measure of the controlled process and that of the Wiener reference measure. Our investigation has special relevance for KL-weighted optimal control in `open-loop' settings where the control is not state dependent, such as in the control of large ensembles of particles, such as fleets of vehicles, modeling flocking behavior~\cite{brockett2012notes}, material science applications~\cite{whitelam2020learning,tang2017comparison} and in power systems~\cite{chertkov2017ensemble, chertkov2018ensemble,metivier2020mean} (for load management for example). 

The utility of the KL-weighted control formulation comes from a variational formula proved in~\cite{Bierkens} (and in~\cite{BoueDupuis} in a more restricted setting) that shows that the optimal value of the KL-weighted control is precisely the logarithm of the mean exponentiated cost under the Wiener measure.  However, even though the optimal value (and corresponding optimal measure) can be characterized, computing the optimal control function itself is still a non-trivial problem. For example,~\cite{Bierkens} shows that the optimal control function can be computed as the F\"ollmer drift~\cite{follmer1985entropy,follmer1986time,lehec2013representation} corresponding to the optimizer of the information projection. In most cases, the computation of the F\"ollmer drift is fraught; for instance,~\cite{Bierkens} uses the Clark-Occone formula of Malliavin calculus to compute the optimal control function. 

\paragraph{\bf Organization.} In the next section we present preliminary notation, definitions, and standard results that will be used throughout the paper. In Section~\ref{sec:OM}, we formalize our information projection problem and show that it is equivalent to minimization of an Onsager-Machlup function. In Section~\ref{sec:rewoc} we introduce the constrained relative entropy-weighted optimization problem and connect it to Onsager-Machlup minimization. We then characterize the constraint that the drift is state-independent in Section~\ref{sec:char}. Specifically, we show that this constraint can be expressed by a single functional constraint (which is a difference-of-convex (DC) functions), and we consider some convex relaxations. In Section~\ref{sec:el} we turn to the Euler-Lagrange formulation of the problem and provide a clean derivation of the optimal solution in terms of the solution to the Euler-Lagrange equation. We also offer an interpretation of the solution to this problem in terms of minimizing the Onsager-Machlup function that was introduced in Section~\ref{sec:OM}. Finally, in Section~\ref{sec:examples} we compute the solution of the Euler-Lagrange equation and the associated Onsager-Machlup minimizing process for a number of example cost functions. We conclude the paper in Section~\ref{sec:conclusion} with some further discussion.

\section{Preliminaries}
In this section we collect several classical results that are frequently referenced throughout this paper. First, we discuss results for Gaussian measure theory on Banach spaces. We refer the reader to \cite{Bogachev,hairer2009introduction} for more information on Gaussian measure theory.

We will work on a separable Banach space $(\mathcal B,\,\|\cdot\|)$, where we denote elements of $\mathcal B$ as $\omega \in \mathcal B$. We let $\mathcal B^\ast$ denote the space of continuous linear functionals $\ell:\mathcal B\to \mathbb R$ (with respect to the norm on $\mathcal B$).

\begin{definition}\label{def:Gaussian-measure}\cite[Definition 3.2]{hairer2009introduction}
Let $\mu$ be a Borel probability measure on $\mathcal B$.

(i) The \textbf{pushforward measure} $\ell^\ast(\mu)$ for $\ell \in \mathcal B^\ast$ is defined by $\ell^\ast(\mu)(A):=\mu(\ell^{-1}(A))$ for all Borel $A\subset \mathbb R$.

(ii) $\mu$ is \textbf{Gaussian} if, for all $\ell \in \mathcal B^\ast$, $\ell^\ast (\mu)$ is a Gaussian measure on $\mathbb R$.

(iii) $\mu$ is \textbf{centered} if, for all $\ell \in \mathcal B^\ast$, $\ell^\ast(\mu)$ is centered.
\end{definition}

Formally, we consider Dirac-$\delta$ measures to be Gaussian in the limit of infinitesimally small variance. This convention includes the case where e.g. $\ell\equiv 0$. 

\begin{definition}\label{def:cov-operator}\cite[Equation (3.2)]{hairer2009introduction} 
Let $\mu$ be a centered Gaussian measure on $\mathcal B$. The \textbf{covariance operator} corresponding to $\mu$ is $C_\mu:\mathcal B^\ast\times \mathcal B^\ast \to \mathbb R$ defined by
\begin{equation*}
    C_\mu (\ell, \ell'):= \int_{\mathcal B} \ell(x)\ell'(x) \mu(dx),
\end{equation*}
for all $\ell,\,\ell' \in \mathcal B^\ast$.

\end{definition}
\noindent When $\mathcal B=\mathbb R^d$, we may identify $(\mathbb R^d)^\ast$ with $\mathbb R^d$ so that the covariance operator from Definition \ref{def:cov-operator} is just the usual covariance matrix.

Next we define the Cameron-Martin space which plays a central role in this paper. In particular, in our upcoming information projection problem we will optimize over all shift measures corresponding to the Cameron-Martin space.

\begin{definition}\label{def:CM-space}\cite[Definition 3.24]{hairer2009introduction}
Let $\mu$ be a centered Gaussian measure on $\mathcal B$. The \textbf{Cameron-Martin space} $\mathcal H_\mu$ is the completion of the subspace
\begin{equation*}
    \mathring{\mathcal H}_\mu := \{h\in \mathcal B: \exists h^\ast \in \mathcal B^\ast \text{ with }C_\mu (h^\ast, \ell) =\ell(h),\, \forall \ell \in \mathcal B^\ast \}
\end{equation*}
with respect to the norm $\|\cdot\|_\mu$ corresponding to $\mu$ defined by $\|h\|_\mu :=\langle h,h\rangle_\mu :=C_\mu(h^\ast, h^\ast)$ for all $h \in \mathcal B$.

\end{definition}

\begin{remark}
There might be multiple $h^\ast$ associated to $h$ in the sense of Definition \ref{def:CM-space}.  However, the norm $\|\cdot\|_\mu$ is independent of the choice of $h^\ast$. Furthermore, there is a canonical representation for $h^\ast$ (see e.g. \cite[Proposition 3.31]{hairer2009introduction}). In the rest of our work, we just assume $h^\ast$ is this canonical representation. 
\end{remark}

One of the main achievements of Gaussian measure theory is the Cameron-Martin theorem. For $h\in \mathcal B$, the translation map $T_h:\mathcal B\to \mathcal B$ is defined by $T_h(x):=x+h$ for all $x \in \mathcal B$. Let $\mu_0$ be a reference Gaussian measure on $\mathcal B$. Corresponding to the translation map $T_h$, we define the pushforward measure $\mu:=T_h^\ast(\mu_0)$  by $\mu(A)=T_h^\ast(\mu_0)(A)=\mu_0(T_h^{-1}(A))$ for all Borel $A$.

\begin{theorem}\label{theorem:CM} \cite[Theorem 3.41]{hairer2009introduction}
Let $\mu_0$ be a centered Gaussian measure on $\mathcal B$.  The pushforward measure $\mu:=T_h^\ast(\mu_0)$ is absolutely continuous with respect to $\mu_0$ if and only $h\in \mathcal H_{\mu_0}$. Furthermore, the Radon-Nikodym derivative of $\mu$ with respect to $\mu_0$ is
\begin{equation*}
    \frac{d\mu}{d\mu_0}=\exp \left(h^\ast(x)-\frac{1}{2}\|h\|_\mu^2\right).
\end{equation*}
\end{theorem}

We now recall the Kullback-Leibler (KL) divergence, which is central to our information projection problem (it is the objective of this problem). Continue to let $\mu_0$ be a reference measure on $\mathcal B$ and let $\mu$ be another measure on $\mathcal B$ such that the Radon-Nikodym derivative $\frac{d\mu}{d\mu_0}$ exists. Then, the KL-divergence of $\mu$ with respect to $\mu_0$ is defined by:
\[
D_{KL}(\mu||\mu_0):=E_{\mu}\left[\log\left(\frac{d\mu}{d\mu_0}\right)\right].
\]
We recall three key properties of the KL-divergence: (i) $D_{KL}(\mu||\mu_0) \geq 0$; (ii) $D_{KL}(\mu||\mu_0) = 0$ if and only if $\mu = \mu_0$; and (iii) $\mu \mapsto D_{KL}(\mu||\mu_0)$ is convex. The next result shows how to compute the KL-divergence between two shift measures drawn from the Cameron-Martin space $\mathcal H_{\mu_0}$.

\begin{theorem}\label{theorem:KL-for-shift}\cite[Lemma 3.20]{Dan-KLD-CM}
Let $\mu_0$ be a centered Gaussian measure on $\mathcal B$ with Cameron-Martin space $\mathcal H_{\mu_0}$, and let $\mu_1 = T_{h_1}^\ast(\mu_0)$ and $\mu_2 = T_{h_2}^\ast(\mu_0)$ for some $h_1,h_2\in \mathcal H_{\mu_0}$. Then, the Radon-Nikodym derivative $\frac{d\mu_1}{d\mu_2}$ exists and
\begin{equation*}
    D_{KL}(\mu_1||\mu_2) = \frac{1}{2}\|h_1-h_2\|_{\mu_0}^2.
\end{equation*}
\end{theorem}

Next we define classical Wiener space associated to a standard Brownian motion. This is our first specific example of a Banach space with a Gaussian measure, and later in the paper we obtain some more detailed results for this special setting.

\begin{definition}\label{def:Wiener-space}
Let $\mathcal B=\mathcal C_0[0,T]$ be the set of continuous functions $f:[0,T]\to \mathbb R$ equipped with the supremum norm $\|\cdot\|_\infty$, such that $f(0)=0$. Let $\mu_0$ be the Borel measure on $\mathcal C_0[0,T]$ associated to the standard Brownian motion $B(t)$. Then $(\mathcal C_0[0,T], \mu_0)$ is \textbf{classical Wiener space}. 
\end{definition}

\begin{proposition}\label{prop:GM-for-WS}\cite[Exercise 3.27]{hairer2009introduction} 
Let $\mu_0$ be the Borel measure on $\mathcal C_0[0,T]$ associated to a standard Brownian motion $B(t)$.

(i) The measure $\mu_0$ is a centered Gaussian measure with covariance operator $C_{\mu_0}(\delta_s,\delta_t)=\min\{s,\, t\}$.

(ii) The Cameron-Martin space associated to $\mu_0$ is the Sobolev space $$W_0^{1,2}:=\left\{F:[0,T]\to \mathbb R: f(0)=0\text{ and } \exists f\in L^2[0,T] \text{ so that } F(t)=\int_0^t f(s) ds\right\}.$$ Furthermore, the Cameron-Martin norm $\|\cdot\|_{\mu_0}$ is the Sobolev norm
\begin{equation*}
    \|F\|_{\mu_0}=\int_0^T f^2(s) ds.
\end{equation*}
\end{proposition}

Now we collect some results from stochastic analysis and Malliavin calculus. We refer the reader to \cite{NuaBook,ProtterBook} for more details.

\begin{theorem}\label{theorem:Girsanov1}\cite[Lemma 5.76]{Baudoin}
Let $(\mathcal C_0[0,T], \mu_0)$ be classical Wiener space, and assume:

(a) $F$ is a progressively measurable process with respect to the filtration generated by the standard Brownian motion $B(t)$;

(b) The sample paths are almost surely in the Cameron-Martin space $W_0^{1,2}$;

(c) Novikov's condition,
\begin{equation*}
    E_{\mu_0}\left[\exp \left(\int_0^T f^2(s) ds\right)\right]<\infty,
\end{equation*}
where $f(s) = F'(s)$, holds.

Then, the process $B(t)-F(t)$ is a standard Brownian motion under $\mu$ with density
$$\frac{d\mu}{d\mu_0}=\exp\left(\int_0^T f(s)dB(s)-\frac{1}{2}\int_0^T f^2(s) ds\right).$$
\end{theorem}

\begin{theorem}\label{theorem:Girsanov2} \cite[Theorem 5.72]{Baudoin}
Let $(\mathcal C_0[0,T], \mu_0)$ be classical Wiener space, and suppose $\mu\sim\mu_0$. Then, there exists a progressively measurable process $F(t)$ with sample paths almost surely in the Cameron-Martin space $W_0^{1,2}$ so that the process $B(t)-F(t)$ is a standard Brownian motion under $\mu$. Furthermore, the density is given by
$$\frac{d\mu}{d\mu_0}=\exp\left(\int_0^T f(s)dB(s)-\frac{1}{2}\int_0^T f^2(s) ds\right),$$
where $f(s) = F'(s)$.
\end{theorem}

\begin{remark}
Theorem \ref{theorem:Girsanov2} is a partial converse to Theorem \ref{theorem:Girsanov1}. We had to assume Novikov's condition holds to prove the ``forward" direction in Theorem \ref{theorem:Girsanov1}. There are weaker sufficient conditions, such as Kazamaki's condition (see e.g. \cite[P. 331]{Revuz}). However in the ``reverse" direction in Theorem \ref{theorem:Girsanov2}, the conclusion only establishes that the sample paths lie almost surely in the Cameron-Martin space $W_0^{1,2}$. This condition is weaker than both Novikov's condition and Kazamaki's condition. 
\end{remark}

\begin{remark}
By the Doob-Dynkin Lemma, the progressively measurable process $f$ that appears in Theorem \ref{theorem:Girsanov1} and Theorem \ref{theorem:Girsanov2} has implicit dependence on the underlying Brownian motion, i.e., $f(s) = f(s,\, B(s))$ for all $s \in [0,T]$. We usually suppress this dependence for cleaner notation except where it is needed explicitly.
\end{remark}
In light of Theorem \ref{theorem:Girsanov2} we make the following definition.
\begin{definition}\label{def:corresponds}
Let $\mu \sim \mu_0$ be a Borel measure on classical Wiener space $(C_0[0,T],  \mu_0)$. We say that the progressively measurable process $F(t)$ \textbf{corresponds to} the measure $\mu$ if $$\tilde B(t):=B(t)-F(t)$$ is a Brownian motion under $\mu$.
\end{definition}

\begin{theorem}\label{theorem:Ito-rep}\cite[Theorem 5.55]{Baudoin}
Let $(\mathcal C_0[0,T], \mu_0)$ be classical Wiener space, and let $C:\mathcal C_0[0,T]\to \mathbb R$ be a functional such that $E_{\mu_0}[C^2]<\infty$. Then, there exists a progressively measurable process $f(t)$ such that 
$$C=E_{\mu_0}[C]+\int_0^T f(t)dB(t).$$
\end{theorem}

Theorem \ref{theorem:Ito-rep} is an existence result. A natural follow-up question is how to compute the process $f(t)$ that appears in the statement of Theorem \ref{theorem:Ito-rep}. The Clark-Ocone theorem gives a computational version of Theorem \ref{theorem:Ito-rep}, stated in terms of the Malliavin derivative of $C$. We refer to \cite{NuaBook} for more information on Malliavin calculus.

\begin{theorem}\label{theorem:Clark-Ocone}
\cite[Proposition 1.3.14]{NuaBook} Let $C:\mathcal C_0[0,T]\to \mathbb R$ be in the Sobolev-Watanabe space $\mathbb D^{1,2}$ (see \cite[page 27]{NuaBook} for the definition), then 
\begin{equation*}
    C=E_{\mu_0}[C]+\int_0^T E_{\mu_0}[D_t C \mid \mathcal F_t]dB(t),
\end{equation*}
where $D_t$ is the Malliavin derivative and $\mathcal F_t$ is the filtration generated by Brownian motion. 
\end{theorem}

\section{Information Projection and Onsager-Machlup}\label{sec:OM}
For this section, let $\mathcal B$ be a separable Banach space with centered Gaussian measure $\mu_0$. We formally define the set of Gaussian shift measures to be
\begin{equation}\label{eq:P-set-def}
    \mathcal P:=\{T_h^\ast(\mu_0):h\in \mathcal H_{\mu_0}\},
\end{equation}
where the measure $T_h^\ast(\mu_0)$ is defined before Theorem \ref{theorem:CM}. Now let $\mu^\ast$ be another Borel probability measure on $\mathcal B$ such that $\mu^\ast \sim \mu_0$. Recall the information projection problem 
\begin{equation*}
    \mathbb K:=\inf_{\mu\in \mathcal P} D_{KL}(\mu||\mu^\ast).
\end{equation*}

Problem $\mathbb K$ projects $\mu^\ast$ onto the set of Gaussian shift measures. We are interested in finding the measure $\mu_O$ that attains the optimal value $\operatorname{val}(\mathbb K)$ of Problem $\mathbb K$, when it exists.

We will show that computing the optimizer of Problem $\mathbb K$ is equivalent to minimizing the Onsager-Machlup function corresponding to an associated stochastic process. We now introduce the definition of the Onsager-Machlup function, see~\cite{Stuart-OM,Durr} for further details. 

\begin{definition}
Let $\mu_0$ be a Gaussian measure on a separable Banach space $\mathcal B$ with Cameron-Martin space $\mathcal H_{\mu_0}$. Let $B^\delta(z)\subset \mathcal B$ be the open ball of radius $\delta$ around $z$. Let $\mu$ be another measure that is absolutely continuous with respect to $\mu_0$. If the limit
\begin{equation*}
    \lim_{\delta\to 0}\frac{\mu(B^\delta(z_2))}{\mu(B^\delta(z_1))}=\exp\left(OM(z_1)-OM(z_2)\right)
\end{equation*}
exists for all $z_1,z_2\in \mathcal H_{\mu_0}$, then $OM$ is called the \textbf{Onsager-Machlup}  function for $\mu$. 
\end{definition}
\noindent The Onsager-Machlup function can be viewed as the Lagrangian for the ``most likely'' path of the associated process. In Theorem~\ref{theorem:OM-minimization}, we specifically show that the optimizer of Problem $\mathbb K$ can be identified with an optimal drift/shift function that is the minimizer of an associated Onsager-Machlup function. Crucially, this identification affords a Monte Carlo sampling method (via the associated stochastic process) for solving Problem $\mathbb K$.

In \cite[Theorem 3.2]{Stuart-OM}, the Onsager-Machlup function is computed for a certain class of measures. We recall this result next.

\begin{hypothesis}\label{hypothesis:OM}
Let $\Phi:\mathcal B\to \mathbb R$ be a functional satisfying the following conditions:

(i) For every $\varepsilon>0$ there is an $M\in \mathbb R$ such that
$$\Phi(\omega)\geq M-\varepsilon \|\omega\|^2,$$
for all $\omega\in \mathcal B$.

(ii) $\Phi$ is locally bounded above, i.e., for every $r>0$ there is a $K=K(r)>0$ such that
$$\Phi(\omega)\leq K,$$
for all $\omega\in\mathcal B$ with $\|\omega\|<r$.

(iii) $\Phi$ is locally Lipschitz continuous, i.e., for every $r>0$ there exists $L=L(r)>0$ such that
$$|\Phi(\omega_1)-\Phi(\omega_2)|\leq L\|\omega_1-\omega_2\|,$$
for all $\omega_1,\omega_2\in \mathcal B$ with $\|\omega_1\|<r$ and $\|\omega_2\|<r$.
\end{hypothesis}
\begin{remark}
Note that the functional $\Phi$ in Hypothesis \ref{hypothesis:OM} is not necessarily of the form in Theorem \ref{theorem:Portmanteau}, where the latter is given by some real-valued functional $C$. 
\end{remark}

\begin{theorem}\label{theorem:Stuart-OM}
\cite[Theorem 3.2]{Stuart-OM} Let $\Phi:\mathcal B\to \mathbb R$ satisfy Hypothesis \ref{hypothesis:OM}. Let the measure $\mu$ have the density  $$\frac{d\mu}{d\mu_0}=\frac{e^{-\Phi}}{E_{\mu_0}[e^{-\Phi}]}.$$ Then, the OM function is given by
\begin{equation*}
    OM_\Phi(z)=\begin{cases}
    \Phi(z)+\frac{1}{2}\|z\|_{\mu_0}^2 &\text{ if } z\in \mathcal H_{\mu_0}\\
    \infty &\text{ else}
    \end{cases}.
\end{equation*}
\end{theorem}
\noindent We introduce the following optimization problem
$$\mathbb M:=\inf_{z\in \mathcal H_{\mu_0}} OM_\Phi(z),$$
which aims to minimize the OM function $OM_\Phi(z)$. Under Hypothesis \ref{hypothesis:OM}, we know that an optimal solution of Problem $\mathbb M$ exists by \cite[Proposition 3.4]{dashti2013map}. See \cite{Durr} for the corresponding result in classical Wiener space. The minimizer of the OM function $OM_\Phi(z)$ (i.e., the solution of Problem $\mathbb M$) is the ``most likely" element of $(\mathcal B, \mu_0)$. 

To express Problem $\mathbb K$ in terms of the OM function, we need a technical lemma from \cite{Bierkens} for functionals $C$ satisfying the following ``finite entropy" hypothesis.
\begin{hypothesis}\label{hypothesis:FE}
We assume that $\mu_0( C<+\infty)>0$ and $E_{\mu_0}[\exp(-C)\,|C|]<+\infty$.
\end{hypothesis}

\begin{lemma}\label{lemma:Bierkens-General}
\cite[Lemma 2.4(ii)]{Bierkens} Let $\mu_0$ be a Gaussian measure on a separable Banach space $\mathcal B$, and let $C:\mathcal B\to \mathbb R$ satisfy Hypothesis \ref{hypothesis:FE}. Define the measure $\mu^\ast$ with density $$\frac{d\mu^\ast}{d\mu_0}=\frac{e^{-C}}{E_{\mu_0}[e^{-C}]}.$$ Then, for $\mu \ll \mu_0$ satisfying $D_{KL}(\mu||\mu_0)<+\infty$ and $E_\mu [(C)^+]<+\infty$,
\begin{equation*}
    D_{KL}(\mu||\mu^\ast)=E_{\mu}[C]+D_{KL}(\mu||\mu_0)-\log E_{\mu_0}[e^{-C}].
\end{equation*}
\end{lemma}
\begin{remark}
In \cite{Bierkens}, it is shown that Hypothesis \ref{hypothesis:FE} implies $0<E_{\mu_0}[\exp(-C)]<+\infty$, and so the measure $\mu^\ast$ in Lemma \ref{lemma:Bierkens-General} is well defined.
\end{remark}
We can now relate Problem $\mathbb K$ to minimization of an OM function. 
\begin{theorem}\label{theorem:OM-minimization}
Let $\mu_0$ be a centered Gaussian measure on a separable Banach space $\mathcal B$. Let $\mu^\ast$ be another Borel probability measure and assume:

(a) $\mu^\ast \sim \mu_0$.

(b) $\mu^\ast$ has density $$\frac{d\mu^\ast}{d\mu_0}=\frac{e^{-C}}{E_{\mu_0}[e^{-C}]},$$
for a functional $C : \mathcal B \rightarrow \mathbb R$ which satisfies Hypothesis \ref{hypothesis:FE}.

(c) The functional $z \mapsto \Phi(z) :=E_{\mu_0}[C(\omega+z)]$ satisfies Hypothesis \ref{hypothesis:OM}, and there is an associated measure $\tilde \mu$ on $\mathcal B$ with density $$\frac{d\tilde \mu }{d\mu_0}=\frac{e^{-\Phi}}{E_{\mu_0}[e^{-\Phi}]}.$$

Then, Problem $\mathbb K$ has a solution $\mu_O$ where
$\mu_O(\cdot)=\mu_0(\cdot-z^O)$ and $z^O$ is the minimizer of the OM function for $\tilde \mu$ (which always exists by \cite[Prop 3.4]{Stuart-OM}).
\end{theorem}
\begin{proof}
By Lemma \ref{lemma:Bierkens-General}, we have that 
\begin{equation*}
    \inf_{\mu\in \mathcal P} D_{KL}(\mu||\mu^\ast)=\inf_{\mu\in \mathcal P} \left\{ E_{\mu}[C]+D_{KL}(\mu||\mu_0) \right\}-\log E_{\mu_0}[e^{-C}].
\end{equation*}
Using Theorem \ref{theorem:KL-for-shift} and the definition of $\Phi$ for each $\mu\in \mathcal P$ with corresponding shift $z_\mu\in \mathcal H_{\mu_0}$, we have that 
\begin{align*}
   D_{KL}(\mu||\mu^\ast)
    &=E_{\mu_0}[C(\omega+z_\mu)]+\frac{1}{2}\|z_\mu\|_{\mu_0}^2-\log E_{\mu_0}[e^{-C}]\\
    &=\Phi(z_\mu)+\frac{1}{2}\|z_\mu\|_{\mu_0}^2-\log E_{\mu_0}[e^{-C}]\\
    &= OM_\Phi(z_\mu) - \log E_{\mu_0}[e^{-C}].
\end{align*}
By \cite[Prop 3.4]{Stuart-OM}, the minimizer of Problem $\mathbb M$, denoted $z^O\in \mathcal H_{\mu_0}$, exists. Then, we have the chain of relations
\begin{align*}
    OM_\Phi(z^O) - \log E_{\mu_0}[e^{-C}]&=D_{KL}(\mu_{0}(\cdot -z^O)||\mu^\ast)\\
    &\geq \inf_{\mu\in \mathcal P}D_{KL}(\mu||\mu^\ast)\\
    &=\inf_{z\in \mathcal H_{\mu_0}}OM_\Phi(z) - \log E_{\mu_0}[e^{-C}]\\
    &=OM_\Phi(z^O) - \log E_{\mu_0}[e^{-C}].
\end{align*}
It follows that the optimal value $\operatorname{sol}(\mathbb K)$ exists and also that $\operatorname{sol}(\mathbb K)=\mu_0(\cdot -z^O)$. 
\end{proof}

\begin{remark}
Observe that $z^O$ is interpreted as the most likely element of $(\mathcal{B}, \tilde{\mu}) $. In other words, $z^O$ can be viewed as the mode of $\tilde{\mu}$. Theorem \ref{theorem:OM-minimization} suggests a potential simulation-based method to compute the optimizer of the information projection Problem $\mathbb K$ via Monte-Carlo estimation of the mode of $\tilde \mu$. 
\end{remark}

\subsection{KL-Weighted Optimization on Wiener Space}~\label{sec:rewoc}
We now apply our results to KL-weighted state independent (or ``open loop'') optimal control on classical Wiener space. Let $C: \mathcal C_0[0,T]\to \mathbb R$ be a cost functional satisfying Hypothesis \ref{hypothesis:FE} along with the following integrability condition.
\begin{hypothesis}\label{hypothesis:L2}
We assume that $E_{\mu_0}[C^2]<\infty$. 
\end{hypothesis}
\noindent For all $\mu \sim \mu_0$, we define the functional:
\begin{equation}\label{eq:objective}
    J(\mu):= E_\mu \left[C\right] + D_{KL}(\mu || \mu_0) = E_\mu \left[C+\log\left(\frac{d\mu}{d\mu_0}\right)\right],
\end{equation}
which is the sum of the expected cost $E_\mu \left[C\right]$ with respect to $\mu$ and the KL-divergence $D_{KL}(\mu || \mu_0)$ of $\mu$ with resepct to the uncontrolled process $\mu_0$. In particular, $D_{KL}(\mu || \mu_0)$ can be interpreted as a penalty for the ``effort'' of control.

Let
\begin{equation}\label{eq:P0-set-def}
\mathcal P_0 := \{\text{Borel probability measures }\mu:\mu\sim \mu_0\}
\end{equation}
be the set of all measures on $\mathcal B$ that are equivalent to the Wiener measure $\mu_0$. The corresponding classical unconstrained variational stochastic optimal control problem is:
\begin{equation}\label{eq:classic}
    \inf_{\mu \in \mathcal P_0}E_{\mu}\left[C+\log \left(\frac{d\mu}{d\mu_0}\right)\right],
\end{equation}
where we note the feasible region in Problem \eqref{eq:classic} is $\mathcal P_0$ instead of $\mathcal P$ (the set of shift measures defined in Eq. \eqref{eq:P-set-def}). The optimality conditions of Problem \eqref{eq:classic} are characterized in the following theorem, where we denote an optimal solution of Problem \eqref{eq:classic} as $\mu^*$.

\begin{theorem}\label{thm:unconstrained}
\cite[Lemma 2.4]{Bierkens} Suppose Hypothesis \ref{hypothesis:FE} holds. Then, the measure with density
\begin{equation}
    \frac{d\mu^\ast}{d\mu_0}:=\frac{\exp(-C)}{E_{\mu_0}[\exp(-C)]}
\end{equation}
exists and $\mu^\ast\in \mathcal P_0$. In addition, if $\mu$ is a Borel probability measure so that $E_\mu [(C)^+]<+\infty$ and $D_{KL}(\mu||\mu^\ast)<+\infty$, then we have:
\begin{equation}~\label{eq:J}
    J(\mu)=D_{KL}(\mu || \mu^\ast)-\log E_{\mu_0}[e^{-C}].
\end{equation} 
Furthermore, $J$ is strictly convex on the set $$\tilde{\mathcal P}:=\{\text{Borel probability measures } \mu:E_\mu [(C)^+]<+\infty\text{ and } D_{KL}(\mu||\mu^\ast)<+\infty\},$$ and $\mu^\ast$ is the unique optimal solution with corresponding optimal value:
$$\inf_{\mu \in \tilde{\mathcal P}} J(\mu)=-\log E_{\mu_0}\left[e^{-C}\right].$$
\end{theorem}

\noindent By Theorem \ref{theorem:Ito-rep}, if $C$ satisfies Hypothesis \ref{hypothesis:L2} then there exists a progressively measurable process $f(t)$ so that
\begin{equation*}
    C=E_{\mu_0}[C]+\int_0^T f(t) dB(t),
\end{equation*}
where equality holds almost surely. By Theorem \ref{thm:unconstrained}, the density for the optimal measure $\mu^*$ is 
\begin{equation*}
    \frac{d\mu^\ast}{d\mu_0}=\exp\left(-\int_0^T f(t)dB(t)-\frac{1}{2}\int_0^T f^2(t) dt\right), 
\end{equation*}
for this same $f(t)$. Consequently, the optimal measure $\mu^\ast$ corresponds to (in the sense of Definition \ref{def:corresponds}) an optimal drift $F(t)=-\int_0^t f(s) ds$ by Theorem \ref{theorem:Girsanov1}.

As noted before, in this paper, we are interested in minimizing $J(\mu)$ (defined in Eq. \eqref{eq:objective}) subject to the additional constraint that measures correspond to (in sense of Definition \ref{def:corresponds}) a state independent drift. That is, $F(t)$ should be a deterministic path that does not depend on the underlying process $B(t)$. This requirement may be equivalently viewed as constraining the covariance of the controlled process to be the same as that of Brownian motion. In this light, we may express the feasible region $\mathcal P$ (originally defined in Eq. \eqref{eq:P-set-def}) equivalently as:
\begin{equation}\label{eq:P}
    \mathcal P:=\{\mu \in \mathcal P_0: \mu \text{ corresponds to a state independent drift}\}.
\end{equation}
The corresponding constrained optimal control problem is:
\begin{equation*}
    \mathbb P := \inf_{\mu \in \mathcal P} J(\mu) = \inf_{\mu \in \mathcal P} \left\{ E_\mu \left[C\right] + D_{KL}(\mu || \mu_0) \right\}.
\end{equation*}
In the next section we provide a more precise description of the feasible region $\mathcal P$. In particular, we will show that while $J$ is a strictly convex function, the feasible region $\mathcal P$ is not convex, and so Problem $\mathbb P$ is not a convex optimization problem.

Using Lemma \ref{lemma:Bierkens-General}, we can reformulate this constrained optimal control problem as:
\begin{equation}\label{eq:P-reformulation}
    \mathbb P \equiv \inf_{\mu \in \mathcal P} J(\mu) = \inf_{\mu \in \mathcal P} D_{KL}(\mu\|\mu^*) + \log E_{\mu_0}[e^{-C}],
\end{equation}
which reveals the familiar Problem $\mathbb K$. We will use the form of Eq. \eqref{eq:P-reformulation} to connect with Problem $\mathbb M$. Theorem~\ref{theorem:OM-minimization} shows that the optimal solution of Problem $\mathbb K$ is the information projection of the optimal measure $\mu^\ast$ of Problem \eqref{eq:classic} onto $\mathcal P$. By Eq. \eqref{eq:P-reformulation}, we see that the optimal solution of Problem $\mathbb P$ coincides with the optimal solution of Problem $\mathbb K$.

This information projection corresponds to finding the most likely path of an associated process $\tilde X$ (defined through its Onsager-Machlup function). This is philosophically similar to the equivalence of {\it maximum a posteriori} (MAP) inference and information projection in Euclidean spaces. For more information on this relation, see \cite{Joyce2011}. We also observe that the information projection in Problem $\mathbb K$ is intimately connected with MAP estimation and identifying the ``mode'' of the Gaussian measure on a general Banach space~\cite{Stuart-OM}.

We now explicitly reformulate Problem $\mathbb P$ as an optimization problem on path space, and compute the associated process $\tilde X$. Consider the process $\tilde X(t)$ whose Girsanov density is given by
\begin{equation}\label{eq:tilde-mu}
\frac{d\tilde \mu}{d\mu_0}(\omega)=\frac{1}{Z_{\tilde \mu}}\exp\left(-E_{\mu_0}[C(B+\omega)]\right),
\end{equation}
where $Z_{\tilde \mu}$ is a normalizing constant. By Theorem~\ref{theorem:Stuart-OM}, the Onsager-Machlup function associated to $\tilde{\mu}$ is
\begin{equation}\label{eq:OM-for-tilde-X}
OM(\varphi(t))=\begin{cases}E_{\mu_0}\left[C(B+\varphi)\right]+\frac{1}{2}\int_0^T (\dot \varphi(t))^2 dt&\text{ if } \varphi \in W_0^{1,2},\\
\infty&\text{ if }\varphi\not\in W_0^{1,2}.\end{cases}
\end{equation}
This observation leads directly to the following optimization problem on path space:
\begin{equation}\label{eq:path-space}
    \inf_{\varphi \in W_0^{1,2}}E_{\mu_0}[C(B+\varphi)]+\frac{1}{2}\int_0^T (\dot \varphi (t))^2 dt.
\end{equation}

Next we use Girsanov to characterize the process $\tilde X(t)$. We will use $|_{y=x}$ to denote evaluation of $y$ at the point $x$. That is, for a function $g$ we define
\begin{equation*}
    g(y)\bigg|_{y=x}:=g(x).
\end{equation*}
\begin{proposition}~\label{prop:om}
Let $C=C_0-\int_0^T f(t,B(t))dB(t)$ be the It\^o representation of a functional $C:C_0[0,T]\to \mathbb R$ satisfying Hypotheses \ref{hypothesis:FE} and \ref{hypothesis:L2}, where $C_0$ is a constant. Assume that $\Phi$ satisfies Hypothesis \ref{hypothesis:OM}, where $\Phi(z):=E_{\mu_0} [C(\omega+z)]$. Then, the process associated to the measure $\tilde \mu$ defined in Eq. \eqref{eq:tilde-mu} is: 
\begin{equation}\label{eq:tilde-X}
    \tilde X(t)=\int_0^t E_{\mu_0}[f(s,B(s)+x)]\bigg|_{x=B(s)}ds+B(t).
\end{equation}
\end{proposition}
\begin{proof}
Let $\varphi\in C_0[0,T]$, then
\begin{align*}
    E_{\mu_0}[C(B+\varphi)]&=C_0-E_{\mu_0}\left[\int_0^T f(t,B(t)+\varphi(t))dB(t)+\int_0^T f(t,B(t)+\varphi(t))d\varphi(t)\right]\\
    &=C_0-\int_0^T E_{\mu_0}[f(t,B(t)+\varphi(t))]d\varphi(t).
\end{align*}
Therefore, the density of $\tilde \mu$ is 
\begin{equation}
    \frac{d\tilde \mu}{d\mu_0}=\frac{1}{Z_{\tilde \mu}} \exp\left(\int_0^T E_{\mu_0}[f(t,B(t)+x)]\bigg|_{x=B(t)} dB(t)\right),
\end{equation}
where we absorbed the constant $C_0$ into the constant $Z_{\tilde \mu}$. The conclusion then follows by Theorem \ref{theorem:Girsanov1}.
\end{proof}


\begin{remark}
    Observe that the equivalent path space problem in Eq. \eqref{eq:path-space} formally parallels the variational representation for the classical unconstrained KL-weighted control Problem~\eqref{eq:classic} from~\cite{BoueDupuis,Bierkens}, where it is shown that
    \[
     -\log E_{\mu_0}[e^{-C}] = \inf_{f \in A} E_{\mu_0} \left[C(B+f) + \frac{1}{2} \int_0^T f^2(t) dt \right],
    \]
    where $A$ is the set of all progressively measurable processes with respect to the natural filtration of the Wiener process $B$.
\end{remark}

\begin{remark}
Interestingly enough, our identification of the optimizer of Problem $\mathbb{P}$ with the most likely path of $\tilde{X}$ as defined in Eq. \eqref{eq:tilde-X} (in the sense of minimizing the Onsager-Machlup function) is closely related to~\cite[Theorem 5]{todorov2011finding}. The latter result identifies the optimal trajectory of a (deterministic) optimal control problem with the most likely path (in the Onsager-Machlup sense) of a related optimally controlled stochastic process.

Specifically,~\cite{todorov2011finding} considers cost functionals of the form $C(B+f) = \int_0^T g(B(t)+f(t), f(t)) dt + G(B(T))$, with cost rate $g(x,u) = \ell(x,u) + \frac{1}{2} |u|^2$. This total cost model corresponds precisely with the KL-weighted optimal control setting under consideration in this section. Our focus, however, is entirely on state independent, open loop control where the optimal control process/drift coincides with the most likely path of $\tilde X$.  
\end{remark}

The results in this section require the verification of Hypotheses~\ref{hypothesis:OM} and~\ref{hypothesis:FE}, which are non-trivial to satisfy even on classical Wiener space. Consequently, in the next sections, we delve further into the Wiener space setting and develop alternative characterizations of the optimal solution of the information projection problem.

\section{Characterization of the Feasible Region in Wiener Space}~\label{sec:char}
We can develop a finer characterization of Problem $\mathbb K$ when $(\mathcal B, \mu_0)$ is classical Wiener space (see Definition \ref{def:Wiener-space}), which we present in this section. In particular, we construct a penalty function corresponding to state independent drifts. This will let us encode the feasible region $\mathcal P$ through a single functional constraint on the space of measures on continuous paths in $\mathcal C_0[0,T]$. In particular, we show that this constraint function is a ``difference of convex'' (DC) functions that is zero if the measure corresponds to a state independent drift, and that is positive otherwise.



Define the function $D:\mathcal P_0\to [0,\infty)$ via
\begin{equation}\label{eq:constraint}
D(\mu):=D_{KL}(\mu||\mu_0)-\frac12 \int_0^T \left(\partial_s E_\mu [B(s)]\right)^2 ds.
\end{equation}
We interpret $D$ as a penalty function which measures violation of the requirement that the drift be state independent, and in particular we will show that $D(\mu)=0$ if and only if $\mu \in \mathcal P$. The following lemma confirms that $D$ correctly checks if $\mu$ corresponds to a state independent drift. The proof uses the idea that the variance of the drift is zero only when it is state independent.
\begin{proposition}\label{lem:D lemma}
For all $\mu \in \mathcal P_0$, $D(\mu) \geq 0$. Additionally, $D(\mu)=0$ if and only if $\mu$ corresponds (in the sense of Definition \ref{def:corresponds}) to a state independent drift. 
\end{proposition}
\begin{proof}
Fix a shift measure $\mu \in \mathcal P_0$. By Theorem \ref{theorem:Girsanov2}, there exists a progressively measurable process $F$ whose sample paths lie almost surely in the Sobolev space $W_0^{1,2}$ such that $$\log  \left(\frac{d\mu}{d\mu_0}\right)=\int_0^T f(s)\ dB(s)-\frac12 \int_0^T f^2(s) ds,$$ where $f(s) = F'(s)$. In addition, we have $B(t)=F(t)+\tilde B(t)$, where $\tilde B$ is a Brownian motion under $\mu$. We may then compute the KL-divergence
\begin{align*}
    D_{KL}(\mu||\mu_0)&=E_\mu \left[\int_0^T f(s)\ dB(s)-\frac12 \int_0^T f^2(s) ds\right]\\
    &= E_\mu \left[\int_0^T f(s) d(\tilde B(s)+F(s))-\frac12 \int_0^T f^2(s) ds\right]\\
    &=E_\mu \left[\int_0^T f(s) d\tilde B(s)+\frac{1}{2}\int_0^T f^2(s) ds\right].
\end{align*}
An It\^o integral has mean zero, and so 
\begin{equation} \label{eq:D-func-proof}
D_{KL}(\mu||\mu_0)=\frac12 E_\mu \left[\int_0^T f^2(s)\ ds\right].
\end{equation}
Substituting Eq. \eqref{eq:D-func-proof} into Eq. \eqref{eq:constraint}, we arrive at
\begin{equation*}
D(\mu)=\frac12 E_\mu \left[\int_0^T f^2(s)\ ds\right] -\frac12 \int_0^t (\partial_s E_\mu[F(s)])^2\ ds. 
\end{equation*}
To write the second integral, we used the fact that $B(t)=F(t)+\tilde B(t)$, where $\tilde B$ is a Brownian motion under $\mu$. By applying Fubini's theorem and differentiating under the integral sign, we obtain
\begin{equation*}
D(\mu)=\frac12\int_0^T \left(E_\mu [f^2(s)]-(E_\mu [f(s)])^2 \right)ds=\frac12\int_0^T \operatorname{Var}_\mu(f(s))ds.
\end{equation*}
It follows that $D(\mu)=0$ if and only if $\operatorname{Var}_\mu(f(s))=0$ for all $s\in [0,T]$, i.e., $F$ is deterministic. 
\end{proof}
\begin{remark}
Constructing the penalty function $D$ is nontrivial. In general, given an arbitrary shift measure $\mu \in \mathcal P$, where $\mathcal P$ is defined in \eqref{eq:P}, it is difficult to construct the drift $F_\mu(t)$ corresponding to $\mu$. It is even difficult to distinguish between $F(t)$ and $-F(t)$ for deterministic drifts $F\in W_0^{1,2}$. To see this, note that the Girsanov density for $F(t)$ is 
\begin{equation*}
    \frac{d\mu^{(F)}}{d\mu_0}=\exp\left(\int_0^T f(t) dB(t)-\frac{1}{2}\int_0^T f^2(t) dt\right),
\end{equation*}
where $f(t) = F'(t)$, while the Girsanov density for $-F(t)$ is 
\begin{equation*}
        \frac{d\mu^{(-F)}}{d\mu_0}=\exp\left(\int_0^T -f(t) dB(t)-\frac{1}{2}\int_0^T f^2(t) dt\right).
\end{equation*}
As $f$ is deterministic and therefore $\int_0^T f(t) dB(t)\stackrel{d}{=}-\int_0^T f(t) dB(t)$ in distribution, we must have that
\begin{equation*}
    \frac{d\mu^{(-F)}}{d\mu_0}\stackrel{d}{=}\frac{d\mu^{(F)}}{d\mu_0}.
\end{equation*}
It is thus difficult to ``decouple'' $F(t)+B(t)$ through a polarization identity. 
\end{remark}

Based on Proposition \ref{lem:D lemma}, we can reformulate Problem $\mathbb K$ as:
\begin{equation*}
    \mathbb K \equiv \inf_{\mu \in \mathcal P_0} \left\{D_{KL}(\mu||\mu^\ast) : D(\mu ) \leq 0 \right\},
\end{equation*}
where the constraint $\mu \in \mathcal P$ is now encoded by the single functional constraint $D(\mu ) \leq 0$ and the implicit constraint $\mu \in \mathcal P_0$.

Problem $\mathbb K$ is a non-convex optimization problem because the constraint function $D(\mu)$ is non-convex. Specifically, $D(\mu)$ is a ``difference of convex functions'', i.e., the KL-divergence $D_{KL}(\mu||\mu_0)$ (which is convex) and the function $\frac{1}{2}\int_0^T (\partial_s E_\mu [B(s)])^2 ds$ (which is convex because the function $\partial_s E_\mu [B(s)]$ is linear in $\mu$, and the square of a linear function is convex). So, Problem $\mathbb K$ has a single DC constraint (see \cite{tao1997convex,tao2005dc} for discussions on this class of optimization problems). In this situation, it is natural to consider both inner and outer convex approximations of Problem $\mathbb K$. However, we will now see that the feasible region $\mathcal P$ is pathological from the perspective of convex approximation. 

First we consider an inner approximation. As the following proposition shows, the strict convex combination of measures corresponding to deterministic drifts does not correspond to a deterministic drift.

\begin{proposition}\label{lem:convex combination}
Let $\mu_1\sim \mu_0$ and $\mu_2\sim \mu_0$ be two measures corresponding to drifts $F_1(t)$ and $F_2(t)$ (possibly state dependent) in the sense of Definition \ref{def:corresponds}. Then, for $\lambda \in (0,1)$, the convex combination  $\mu:=\lambda \mu_1+(1-\lambda) \mu_2$ corresponds to a drift of 
\begin{equation*}
    F(t):=A\, F_1(t)+(1-A) F_2(t),
\end{equation*}
where $A$ is a Bernoulli random variable independent of $B(t)$ with parameter $\lambda$.
\end{proposition}
\begin{proof}
Let $B(t)$ be a Brownian motion under $\mu_0$. Then, for each time $t\in [0,T]$ and $z\in \mathbb R$ we have 
\begin{align*}
    \mu\left(\{\omega\in \Omega : B(t,\omega)\leq z\}\right)&=\lambda \mu_1\left(\{\omega\in \Omega : B(t,\omega)\leq z\}\right)\\
    &\quad +(1-\lambda) \mu_2\left(\{\omega\in \Omega : B(t,\omega)\leq z\}\right)\\
    &=\lambda \mu_1\left(\{\omega\in \Omega : F_1(t,\omega)+B^1(t,\omega)\leq z\}\right)\\
    &\quad +(1-\lambda)\mu_2\left(\{\omega\in \Omega : F_2(t,\omega)+B^2(t,\omega)\leq z\}\right),
\end{align*}
where $B^1$ is a Brownian motion under $\mu_1$ and $B^2$ is a Brownian motion under $\mu_2$. Now let $\mu_F$ be the measure that corresponds to the drift of $F(t)=AF_1(t)+(1-A)F_2(t)$. By the law of total probability, for $t\in [0,T]$ and $z\in \mathbb R$ we have that 
\begin{align*}
    \mu_F \left(\{\omega\in \Omega : B(t,\omega)\leq z\}\right)&=\lambda\, \mu_F \left(\{\omega\in \Omega : B(t,\omega)\leq z\}\mid A=1\right)\\
    &\quad+(1-\lambda)\mu_F \left(\{\omega\in \Omega : B(t,\omega)\leq z\}\mid A=0\right)\\
    &=\lambda\, \mu_1\left(\{\omega\in \Omega : F_1(t,\omega)+B^1(t,\omega)\leq z\}\right)\\
    &\quad +(1-\lambda)\mu_2\left(\{\omega\in \Omega : F_2(t,\omega)+B^2(t,\omega)\leq z\right).
\end{align*}
Then, for all $t\in [0,T]$ and $z\in \mathbb R$, we have the equality
\begin{equation*}
    \mu \left(\{\omega\in \Omega : B(t,\omega)\leq z\}\right)=\mu_F\left(\{\omega\in \Omega : B(t,\omega)\leq z\}\right).
\end{equation*}
\end{proof}
\begin{remark}
Proposition \ref{lem:convex combination} implies that the set of measures $\mathcal P$ is precisely the set of extreme points of the set $\mathcal P_0$. If $\mu=\lambda\, \mu_1+(1-\lambda) \mu_2$ for some $\lambda\in (0,1)$ and $\mu_1,\mu_2\in \mathcal P_0$, then the drift associated to $\mu$ is necessarily state dependent (and thus random). Consequently, any inner convex approximation can only consist of singletons, yielding only a trivial inner approximation to Problem $\mathbb K$. 
\end{remark}
\noindent Now we consider outer approximation of $\mathcal P$. An obvious convex relaxation of $\mathcal P$ is to consider the convex hull of $\mathcal P$,  defined as:
$$
\text{conv}\,\mathcal P:=\left\{\sum_{i=1}^n \lambda_i \mu_i : \sum_{i=1}^n \lambda_i =1,\, \lambda \geq 0,\, \{\mu_i\}_{i=1}^n \subset \mathcal P,\, n \in \mathbb N\right\}.
$$
As a consequence of Proposition \ref{lem:convex combination}, we have the following characterization of $\text{conv}\,\mathcal P$.
\begin{corollary}\label{cor:convex hull}
$\text{conv}\,\mathcal P$ is the set of all measures corresponding to drifts of finite \emph{random} combinations of deterministic drifts. That is, $\mu\in \text{conv}\,\mathcal P$ if and only if its corresponding drift $F(t)$ is of the form
\begin{equation*}
    F(t)=
    \begin{cases}
    F_1(t) &\text{ with probability } p_1\\
    &...\\
    F_n(t) &\text{ with probability } p_n
    \end{cases}
\end{equation*}
for some $n\in \mathbb N$, where $F_i$ are deterministic drifts and $p_1+\cdots+p_n=1$.
\end{corollary}
It follows that the convex hull $\text{conv}\,\mathcal P$ is too large to optimize over effectively (especially since we are talking about measures on path space). In light of Proposition \ref{lem:convex combination} and Corollary \ref{cor:convex hull}, we see that the obvious inner and outer convex approximations of Problem $\mathbb K$ do not yield useful approximations for computing the optimal drift (or the corresponding measure).

\section{Euler-Lagrange Equation for Wiener Space}~\label{sec:el}
In response to the difficulties with the non-convex formulation of Problem $\mathbb K$ in Wiener space discussed in the last section, we give another reformulation of the problem in this section. In particular, we show that the solution of a certain calculus of variations problem (when it exists) corresponds to an optimal solution of Problem $\mathbb K$. In many cases this method is more tractable, and, as we demonstrate in several of our examples, this calculus of variations problem is often a convex optimization problem (in contrast to the formulation of Problem $\mathbb K$ just discussed in Section \ref{sec:char}, which is intrinsically non-convex).

The objective of Problem $\mathbb K$ is to minimize $D_{KL}(\mu||\mu^\ast)$ over $\mu \in \mathcal P$ with respect to some $\mu^\ast$. This measure $\mu^\ast$ satisfies $\mu^\ast \sim \mu_0$, so by Theorem \ref{theorem:Girsanov2} there is a progressively measurable $f$ such that:
$$\frac{d\mu^\ast}{d\mu_0}=\exp\left(\int_0^T f(s, B(s))dB(s)-\frac{1}{2}\int_0^T f^2(s, B(s)) ds\right).$$
For this $f$, we define the following Lagrangian for all $q \in W_0^{1,2}[0,T]$:
\begin{equation}\label{eq:lagrangian-definition}
    L(t,q(t),\dot{q}(t);\,f) := E_{\mu_0}\left[(\dot{q}(t)-f(t,B(t)+q(t)))^2 \right].
\end{equation}
The form of this Lagrangian is not a coincidence. In our upcoming Theorem \ref{theorem:main}, we will show that it is a natural consequence of trying to minimize the KL-divergence $D_{KL}(\mu||\mu^\ast)$.

With this Lagrangian now in hand, we pose the corresponding calculus of variations problem:
\begin{equation}\label{eq:calculus_of_variations}
    \inf_{q \in W_0^{1,2}[0,T]} \left\{ \mathcal{L}(q,\dot{q};\,f) := \frac{1}{2} \int_0^T L(t,q(t),\dot{q}(t);\,f)dt \text{ s.t. } q(0) = 0 \right\}.
\end{equation}
\noindent In Problem \eqref{eq:calculus_of_variations}, we note the initial condition $q(0) = 0$ at time $t = 0$, but we do not enforce any terminal condition at time $t = T$. Much of the classical theory for the calculus of variations requires explicit initial and terminal conditions (see e.g. \cite{calder2020calculus}). So, we can simply consider the alternative calculus of variations problem:
\begin{equation}\label{eq:calculus_of_variations-terminal}
    \inf_{q \in W_0^{1,2}[0,T]} \left\{ \mathcal{L}(q,\dot{q};\,f) := \frac{1}{2} \int_0^T L(t,q(t),\dot{q}(t);\,f)dt \text{ s.t. } q(0) = 0,\, q(T) = a \right\},
\end{equation}
with dummy terminal condition $q(T) = a$ for any $a \in \mathbb R$. Then, we can solve Problem \eqref{eq:calculus_of_variations-terminal} for these specific initial and terminal conditions, and later optimize over $a \in \mathbb R$. We will see in our examples in Section \ref{sec:examples} that this issue does not present any practical difficulty.

We first discuss existence of solutions to Problem \eqref{eq:calculus_of_variations}, and we refer the reader to \cite{calder2020calculus} for the details of these existence results. We first require the following coercivity condition:

\begin{hypothesis}\label{assum:coercivity}
There exist $\alpha >0$ and $\beta \geq 0$ such that $L(t,x,\dot x) \geq \alpha |\dot x|^2 - \beta$ for all $t \in [0,T]$, $x \in \mathbb R$, and $\dot{x} \in \mathbb R$.
\end{hypothesis}

For the next existence result, we recall that the Lagrangian $L(t,q,\dot{q};\,f)$ is automatically bounded below by zero by construction (since it is the expectation of a square), and so the optimal value of Problem \eqref{eq:calculus_of_variations} is also bounded below by zero.

\begin{theorem}\label{thm:existence}
\cite[Theorem 4.25]{calder2020calculus} Suppose $L(t,q,\dot{q};\,f)$ satisfies Hypothesis \ref{assum:coercivity}, and that $\dot{q} \to L(t,q,\dot{q};\,f)$ is convex. Then, Problem \eqref{eq:calculus_of_variations} has a solution in $W_0^{1,2}[0,T]$.
\end{theorem}

The Euler-Lagrange equation corresponding to Problem \eqref{eq:calculus_of_variations} is:
\begin{equation}\label{eq:Euler-Lagrange}
    L_q(t,q,\dot{q};\,f)-\frac{d}{dt}L_{\dot{q}}(t,q,\dot{q};\,f)=0,\,\forall t \in [0, T].
\end{equation}
Eq. \eqref{eq:Euler-Lagrange} is a common {\em necessary} condition for the optimal solutions of Problem \eqref{eq:calculus_of_variations}. The next theorem connects the solutions of Problem \eqref{eq:calculus_of_variations} (that can be characterized by Eq. \eqref{eq:Euler-Lagrange}) with the optimal solutions of Problem $\mathbb K$.

\begin{theorem}\label{theorem:main}
Let $\mu^\ast \sim \mu_0$ be a probability measure with density $$\frac{d\mu^\ast}{d\mu_0}=\exp\left(\int_0^T f(s,B(s)) dB(s)-\frac12\int_0^T f^2(s,B(s)) ds\right).$$ Let $q \in W_0^{1,2}[0,T]$ be a solution, if it exists, to Eq. \eqref{eq:Euler-Lagrange}. Then, the measure defined by 
\begin{equation}
\frac{d\mu}{d\mu_0}=\exp\left(\int_0^T \dot{q}(s) dB(s)-\frac{1}{2}\int_0^T(\dot{q}(s))^2 ds\right)
\end{equation}
is an optimal solution of Problem $\mathbb K$.
\end{theorem}
\begin{proof}
Let $$\frac{d\mu}{d\mu_0}=\exp\left[\int_0^T h(s) dB(s)-\frac12\int_0^T h^2(s) ds\right]$$
for some state independent $h(s)\in W_0^{1,2}[0,T]$. We can simplify the ratio
\begin{align*}
    \frac{d\mu}{d\mu^\ast}&=\exp\left(\int_0^T [h(s)-f(s,B(s))] dB(s)-\frac{1}{2}\int_0^T [h^2(s)- f^2(s,B(s))]ds\right)\\
    &=\exp\left(\int_0^T [h(s)-f(s,\tilde B(s)+H(s))] d(\tilde{B}(s)+H(s))-\frac{1}{2}\int_0^T [h^2(s)- f^2(s,\tilde B(s)+H(s))]ds\right),
\end{align*}
where $\tilde B$ is a Brownian motion under $\mu$ and $H'=h$ with $H(0)=0$.

Because $H$ is differentiable, we may write that
\begin{align*}
\frac{d\mu}{d\mu^\ast}&=\exp\bigg(\int_0^T [h(s)-f(s,\tilde B(s)+H(s)) ]d\tilde B(s)\\
&~+\frac{1}{2}\int_0^T [h^2(s)-2h(s)f(s,\tilde B(s)+H(s))+ f^2(s,\tilde B(s)+H(s))] ds\bigg).
\end{align*}
Factoring and then computing the KL-divergence as in the proof of Proposition \ref{lem:D lemma} shows that
\begin{equation}\label{eq:main-proof-last-step}
    D_{KL}(\mu||\mu^\ast)=E_\mu\left[ \log\left(\frac{d\mu}{d\mu^\ast}\right)\right]=\frac{1}{2}E_\mu \left[\int_0^T [h(s)-f(s, \tilde B(s)+H(s))]^2 ds\right].
\end{equation}
Minimizing Eq. \eqref{eq:main-proof-last-step} is equivalent to the Euler-Lagrange Eq. \eqref{eq:Euler-Lagrange}.
\end{proof}

To conclude this discussion, we identify some sufficient conditions for a solution of Eq. \eqref{eq:Euler-Lagrange} to be an optimal solution of Problem \eqref{eq:calculus_of_variations}. This depends on the following joint convexity condition.
\begin{hypothesis}\label{assum:convexity}
The Lagrangian is jointly convex in $(q,\dot q) \to L(t,q,\dot q;\,f)$ for all $t \in [0, T]$.
\end{hypothesis}

\begin{theorem}\label{theorem:EL-sufficient}
\cite[Theorem 4.32]{calder2020calculus} Suppose $L(t,q,\dot q;\,f)$ satisfies Hypothesis \ref{assum:coercivity} and Hypothesis \ref{assum:convexity}. If $q \in W_0^{1,2}[0, T]$ is a weak solution of Eq. \eqref{eq:Euler-Lagrange}, then $q$ is an optimal solution of Problem \eqref{eq:calculus_of_variations}.
\end{theorem}
\noindent We use Theorem \ref{theorem:EL-sufficient} in our examples to justify solving Problem \eqref{eq:calculus_of_variations} by solving the ODE in Eq. \eqref{eq:Euler-Lagrange}.

\section{Examples}~\label{sec:examples}
In this section, we solve several specific instances of Problem $\mathbb P$ where $(\mathcal B, \mu_0)$ is classical Wiener space (see Definition \ref{def:Wiener-space}).

\subsection{Case of $C=\int_0^T f(t,B(t))dB(t)$}

We start in this subsection with the case where the It\^o representation $C=\int_0^T f(t,B(t))dB(t)$ of our cost functional is already known. We may assume $E_{\mu_0}[C]=0$ W.L.O.G. since constants do not change the optimal solution of the minimization problem.

\begin{example}
Let $f(s,B(s))=-B(s)$ so we have $$C=\int_0^T -B(s) dB(s).$$ 
First, we note that $C$ can easily be shown to satisfy Hypotheses \ref{hypothesis:FE} and \ref{hypothesis:L2}. To check that $C$ satisfies Hypothesis \ref{hypothesis:OM}, first we write 
$$C=-\frac{1}{2}B^2(T)+\frac{1}{2}T.$$
Then, we compute $\Phi$ as
$$\Phi(z)=\frac{1}{2}E_{\mu_0}[(B(T)+z(T))^2-T]=\frac{1}{2}z^2(T).$$
To check Hypothesis \ref{hypothesis:OM} Part (i), we just note that for $\varepsilon>0$ and $z\in C_0[0,T]$, we have
$$\Phi(z)\geq 0\geq 0-\varepsilon \|z\|_\infty^2.$$ Thus, Part (i) is satisfied with $M=0$. To check Hypothesis \ref{hypothesis:OM} Part (ii), let $r>0$ and let $K(r)=r^2$. Then, if $\|z\|_\infty<r$ we have $$\Phi(z)=\frac{1}{2}z^2(T)\leq \frac{1}{2}r^2\leq K.$$
To check Hypothesis \ref{hypothesis:OM} Part (iii), let $r>0$ and $L(r)=r$. For $\|z_1\|_\infty <r$ and $\|z_2\|_\infty <r$, we have
\begin{align*}
|\Phi(z_1)-\Phi(z_2)|&=\frac{1}{2}|z_1^2(T)-z_2^2(T)|\\
&=\frac{1}{2}|z_1(T)-z_2(T)||z_1(T)+z_2(T)|\\
&\leq \frac{1}{2}2r \|z_1-z_2\|_\infty\\
&=L\|z_1-z_2\|_\infty.
\end{align*}
Since we have verified Hypothesis \ref{hypothesis:OM}, we know that Problem $\mathbb M$ has a solution. 
The Lagrangian $L$ corresponding to this problem is:
\begin{align*}
L(t,q,\dot q)&=\frac{1}{2}E_{\mu_0}\left[(\dot{q}(t)-(B(t)+q(t)))^2\right]\\
&=\frac{1}{2}E_{\mu_0}\left[\dot{q}^2(t)-2\dot{q}(t)(B(t)+q(t))+(B(t)+q(t))^2\right]\\
&=\frac{1}{2}\left[\dot{q}^2(t)-2q(t)\dot{q}(t)+t+q^2(t)\right]\\
&=\frac{1}{2}\left[(\dot{q}(t)-q(t))^2\right]+\frac{1}{2}t.
\end{align*}
We see that $(q, \dot q) \rightarrow L(t,q,\dot q)$ is jointly convex for all $t \in [0, T]$, and it is bounded below by zero. The functional $\mathcal L$ is minimized at $\dot q(t)=q(t)$ which yields the solution $q(t)=Ce^t$. We additionally require $q(0)=0$, and so we set $C=0$ to find the solution is $q(t)\equiv 0$.

To confirm the optimality of this solution using the Euler-Lagrange equation, we see the derivatives of $L$ are:
\begin{equation*}
    L_q(t,q,\dot q)=-\dot{q}(t)+q(t),
\end{equation*}
and \begin{equation*}
    L_{\dot{q}}(t,q,\dot q)=\dot{q}(t)-q(t).
\end{equation*}
The Euler-Lagrange equation is then
\begin{equation*}
    -\dot{q}(t)+q(t)-\ddot{q}(t)+\dot{q}(t)=0,
\end{equation*}
which has solution $q(t)=C_1 e^t+C_2 e^{-t}$. We require $q(0)=0$, and so we must have $C_1=-C_2$. The minimizer of $L$ is therefore $q(t)=C(e^t-e^{-t})$. As stated in Section \ref{sec:el}, in order to get the constant $C$ we must solve
\begin{equation*}
    \inf_{C} \int_0^T L(t,q,\dot q) dt =\inf_C \int_0^T \left[2C^2 e^{-2t}+\frac{1}{2} t\right] dt,
\end{equation*}
which is minimized when $C=0$. 
Furthermore, we can compute:
\begin{equation*}
    E_{\mu_0}[f(s, B(s)+x)]\bigg|_{x=B(s)}=B(s),
\end{equation*} 
to see that the associated process is:
\begin{equation*}
    \tilde X(t)=\int_0^t B(s) ds+B(t).
\end{equation*}
\end{example}

\begin{example}
Suppose that $\mu^\ast$ (the optimal unconstrained measure) corresponds to a deterministic drift $F(t)$ where $F'(s) = f(s)$ for a deterministic function $f$. The Lagrangian for this problem is:
\begin{align*}
L(t,q,\dot{q})&=\frac{1}{2}E_{\mu_0}\left[\dot{q}^2(t)-2\dot{q}(t)f(t)+f^2(t)\right]\\
&=\frac{1}{2}\left[\dot{q}^2(t)-2\dot{q}(t)f(t)+f^2(t)\right].
\end{align*}
We see that $(q, \dot q) \rightarrow L(t,q,\dot q)$ is jointly convex for all $t \in [0, T]$, and the derivatives of $L$ are:
$$L_q=0,$$
and 
$$L_{\dot q}=\dot q(t)-f(t).$$
The Euler-Lagrange equation is then simply
$$-\frac{d}{dt}(\dot q(t)-f(t))=0,$$
which yields the solution $q(t)=F(t)+C_1 t+ C_2$. We set $C_2=0$ to satisfy the initial condition $q(0)=0$.

Furthermore, we can compute 
\begin{equation*}
    E_{\mu_0}[f(s,B(s)+x)]\bigg|_{x=B(s)}=\tilde{f}(s),
\end{equation*}
to see that the associated process $\tilde X(t)$ is:
\begin{equation*}
    \tilde X(t)=\int_0^t \tilde{f}(s)ds+B(t).
\end{equation*}
\end{example}

\begin{example}
Suppose that $\mu^\ast$ (the optimal unconstrained measure) corresponds to the drift $\int_0^T B^2(s) ds$. The Lagrangian for this problem is then: 
\begin{align*}
    L(t,q,\dot q)&=\frac{1}{2}E_{\mu_0}[\dot q^2(t)-2(B(t)+q(t))^2\dot q(t)+(B(t)+q(t))^4]\\
    &=\frac{1}{2}[\dot q^2(t)-2(t+q^2(t))\dot q(t)+q^4(t)+6tq^2(t)+3t^2].
\end{align*}
The derivatives of $L$ are:
$$L_q(t,q,\dot q)=\frac{1}{2}[-4q(t)\dot q(t)+4q^3(t)+12t q(t)]$$
and 
$$L_{\dot q}(t,q,\dot q)=\frac{1}{2}[2\dot q (t)-2t-2q^2(t)].$$
The resulting Euler-Lagrange equation is:
$$-2q(t)\dot q(t)+2q^3(t)+6tq(t)-(\ddot q(t)-1-2q(t)\dot q(t))=0,$$
which simplifies to the second-order ODE $$2q^3(t)+6tq(t)-\ddot q(t)+1=0.$$

Furthermore, we can compute 
\begin{equation*}
    E_{\mu_0}[f(s,B(s)+x)]\bigg|_{x=B(s)}=s+B^2(s),
\end{equation*}
to see that the associated process $\tilde{X}(t)$ is 
\begin{equation*}
    \tilde X(t)=\int_0^t [s+B^2(s)] ds+B(t).
\end{equation*}
\end{example}

\subsection{Case of $C=\int_0^T g(B(t))dt+G(B(T))$}
In this subsection, we consider cost functionals of the form $C=\int_0^T g(B(t))dt+G(B(T))$. First, we show that such cost functionals satisfy Hypotheses \ref{hypothesis:OM}.
\begin{proposition}
Let $C:\mathcal C_0[0,T]\to \mathbb R$ be a functional of the form $C=\int_0^T g(B(t))dt+G(B(T))$, where: (a) $g$ is $L_g-$Lipschitz continuous; (b) $G$ is $L_G-$Lipschitz continuous; and (c) both $g$ and $G$ are bounded below. Then, $\Phi(z):=E_{\mu_0}[C(\omega+z)]$ satisfies Hypothesis \ref{hypothesis:OM}. 
\end{proposition}
\begin{proof}
First we check Hypothesis \ref{hypothesis:OM} Part (i). As both $g$ and $G$ are bounded below, there is some $M\in \mathbb R$ such that 
$$g(x)\geq M \text{ and } G(x)\geq M,$$
for all $x\in \mathbb R$. Then, for any $\varepsilon>0$ and $z\in \mathcal C_0[0,T]$, we have that
$$\Phi(z)=E_{\mu_0}\left[\int_0^T g(B(t)+z(t))dt+G(B(T)+z(T))\right]\geq MT+T:=M'\geq M'-\varepsilon\|z\|_\infty^2.$$

\noindent Therefore, Part (i) is satisfied with $M'=MT+T$. 
Next, we check that Hypothesis \ref{hypothesis:OM} Part (ii) is satisfied. First, we note that by Lipschitz continuity we have for $t\in [0,T]$ and $z\in W_0^{1,2}$ that 

\begin{align*}
   E_{\mu_0} |g(B(t)+z(t))|&\leq    E_{\mu_0}|g(B(t)+z(t))-g(0)|+|g(0)|\\
   &\leq L_g    E_{\mu_0}|B(t)+z(t)|+|g(0)|\\
   &\leq L_g \sqrt{\frac{2t}{\pi}}+z(t)+|g(0)|\\
   &\leq L_g \sqrt{\frac{2T}{\pi}}+\|z\|_\infty+|g(0)|.
\end{align*}
An analogous chain of inequalities holds for $G$. Thus, we have that 
\begin{align*}
\Phi(z)&=E_{\mu_0}\left[\int_0^T g(B(t)+z(t))dt+G(B(T)+z(T))\right]\\
&\leq TL_g \sqrt{\frac{2T}{\pi}}+|g(0)|+L_G \sqrt{\frac{2T}{\pi}}+|G(0)|+2\|z\|_\infty . \end{align*}
It follows that $\Phi$ is locally bounded above. 
Lastly, we need to check Hypothesis \ref{hypothesis:OM} Part (iii). Let $z_1,z_2\in W_0^{1,2}$, then

\begin{align*}
|\Phi(z_1)-\Phi(z_2)|&= \bigg|\int_0^T E_{\mu_0}(g(B(t)+z_1(t))-g(B(t)+z_2(t))dt\\
&~+E_{\mu_0}(G(B(T)+z_1(T))-G(B(T)+z_2(T))\bigg|\\
&\leq \int_0^T L_g|z_1(t)-z_2(t)|dt+L_G|z_1(T)-z_2(T)|\\
&\leq (L_g T+L_G)\|z_1-z_2\|_\infty.
\end{align*}
Therefore, we find that $\Phi$ is Lipschitz, which implies that it is locally Lipschitz. 
\end{proof}

To handle this class of cost functionals in our framework, we need to first compute the It\^o representation for $C$ by using the Clark-Ocone formula, and then apply Theorem \ref{theorem:main} (see \cite{NuaBook} for more information on Malliavin calculus).

For the following examples, we will need to compute the Malliavin derivative of $C$. We use the following result on the Malliavin derivative of Lebesgue integrals to do this computation.

\begin{lemma}
Let $X_u$ be a progressively measurable process in the Sobolev-Watanabe space $\mathbb D^{1,2}$, then
\begin{equation*}
    D_t \int_0^T X_u du=\int_t^T D_t (X_u) du.
\end{equation*}
\end{lemma}
\begin{proof}
First note that by the dominated convergence theorem, we have that 
$$ D_t \int_0^T X_u du=\int_0^T D_t(X_u)du .$$
Then \cite[Corrolary 1.2.1]{NuaBook} implies that $D_t(X_u)=0$ for $u \leq t$, and hence
$$ \int_0^T D_t(X_u) du=\int_t^T D_t(X_u)du .$$
\end{proof}
\noindent Note that, for the indicator function $\chi_{[0,t]}(s)$, we can rewrite:
$$g(B(t))=g\left( \int_0^T \chi_{[0,t]}(s)dB(s) \right)$$
and
$$G(B(T))=G\left(\int_0^T dB(s)\right).$$ Then, using the above lemma along with the definition of the Malliavin derivative, we have
\begin{equation}
    D_t C=\int_t^T g'(B(u))du+G'(B(T)).
\end{equation}
\noindent Next, in order to use the Clark-Ocone formula, we need to compute the conditional expectations $E_{\mu_0}[D_t C \mid \mathcal F_t]$. We prove the following technical lemma for this purpose.

\begin{lemma}
Let $h : \mathbb R \rightarrow \mathbb R$ be Borel measurable and integrable. Then, 
\begin{equation*}
    E_{\mu_0}[h(B(t))\mid \mathcal F_s]=E_{\mu_0}[h(N_x)]\bigg|_{x=B(s)},
\end{equation*}
where $N_x\sim \mathcal N (x,t-s)$. 
\end{lemma}
\begin{proof}
Let $X$ and $Y$ be integrable random variables, and let $\mathcal G$ be a $\sigma-$algebra. If $X$ is $\mathcal G-$measurable and $Y$ is independent of $\mathcal G$, then we have:
$$E[f(X,Y)|\mathcal G]=E[f(x,Y)]\bigg|_{x=X}.$$ Applying this observation to $h$ with $Y=B(t)-B(s)$, $X=B(s)$, and $f(X,Y)=h(X+Y)$ gives the desired result. 
\end{proof}
\noindent Combining the previous two lemmas with the Clark-Ocone formula, we arrive at the expression
\begin{equation*}
    C=E_{\mu_0}[C]+\int_0^T f(s, B(s)) dB(s),
\end{equation*}
where
\begin{equation}\label{eq:f for prototypical}
    f(s,B(s))=\int_s^T E_{\mu_0}[g'(N_x)]\bigg|_{x=B(s)}du +E_{\mu_0}[G'(N_x)]\bigg|_{x=B(s)}.
\end{equation}
Once we find the It\^o representation for $C=\int_0^T g(B(t))dt+G(B(T))$, we can use Theorem \ref{theorem:main} to solve for the optimal deterministic drift.

We now work out some specific examples for various $g$ and $G$.
\begin{example}
Let $C=\int_0^T B(s) ds$. Then, the function $f$ defined in Eq. \eqref{eq:f for prototypical} is specifically
\begin{align*}
    f(s,B(s))&=\int_s^T E_{\mu_0}[1]\bigg|_{x=B(s)} du\\
    &=T-s,
\end{align*}
and we have:
\begin{equation*}
    \int_0^T B(s) ds=\int_0^T (T-s) dB(s).
\end{equation*}
\end{example}
\begin{example}
Let $C=\int_0^T B^2(s)ds$. Then, the function $f$ defined in Eq. \eqref{eq:f for prototypical} is specifically
\begin{align*}
    f(s,B(s))&=\int_s^T E_{\mu_0}[2(N_x)]\bigg|_{x=B(s)} du\\
    &=\int_s^T 2B(s) du\\
    &=2(T-s) B(s),
\end{align*}
and we have:
\begin{equation}
    \int_0^T B^2(s) ds=E_{\mu_0}[C]+\int_0^T 2(T-s)B(s)dB(s).
\end{equation}
\end{example}
\begin{example}
Let $C=\int_0^T B^3(s)ds$. Then, the function $f$ defined in Eq. \eqref{eq:f for prototypical} is specifically
\begin{align*}
    f(s,B(s))&=\int_s^T E_{\mu_0}[3 (N_x)^2]\bigg|_{x=B(s)} du\\
    &=3\int_s^T [(u-s)+B(s)] du\\
    &=(T-s)^3+3(T-s)B(s),
\end{align*}
and we have:
\begin{equation}
    \int_0^T B^3(s)ds=E_{\mu_0}[C]+\int_0^T [(T-s)^3+3(T-s)B(s)]dB(s).
\end{equation}
\end{example}
In line with the previous three examples, we observe that for any polynomial $g$, we will have a polynomial $f$ in Eq. \eqref{eq:f for prototypical}.

\section{Conclusion}\label{sec:conclusion}

This paper is motivated by the problem of doing information projection on Banach spaces with respect to a Gaussian reference measure. We have shown that, in many cases, this kind of projection can be done effectively. In particular, we show that we can often reformulate this information projection problem as a (convex) OM-functional minimization problem and a (convex) calculus of variations problem. The OM perspective suggests a simulation-based scheme for solving the original problem, and the calculus of variations perspective suggests a discretization scheme for solving the original problem (for solving the Euler-Lagrange equation which is a general second-order ODE).

An important future direction for the results of this paper is to use the information projection of the Wiener measure onto Gaussian shift measures, and the characterization of the optimal drift, as a computational approximation to the ``full" F\"ollmer drift. Further problems include the consideration of more general constraint sets and other types of marginal or path constraints on the information projection problem, and characterize the solutions in those settings.

\bibliographystyle{plain}
\bibliography{References}

\begin{thebibliography}{10}

\bibitem{Baudoin}
Fabrice Baudoin.
\newblock {\em Diffusion processes and stochastic calculus}.
\newblock EMS Textbooks in Mathematics. European Mathematical Society (EMS),
  Z\"{u}rich, 2014.

\bibitem{Bierkens}
Joris Bierkens and Hilbert~J Kappen.
\newblock Explicit solution of relative entropy weighted control.
\newblock {\em Systems \& Control Letters}, 72:36--43, 2014.

\bibitem{Bogachev}
Vladimir~I. Bogachev.
\newblock {\em Gaussian measures}, volume~62 of {\em {Mathematical Surveys and
  Monographs}}.
\newblock American Mathematical Society, Providence, RI, 1998.

\bibitem{BoueDupuis}
Michelle Bou{\'e}, Paul Dupuis, et~al.
\newblock A variational representation for certain functionals of {Brownian}
  motion.
\newblock {\em The Annals of Probability}, 26(4):1641--1659, 1998.

\bibitem{brockett2012notes}
Roger Brockett.
\newblock Notes on the control of the liouville equation.
\newblock In {\em Control of partial differential equations}, pages 101--129.
  Springer, 2012.

\bibitem{calder2020calculus}
Jeff Calder.
\newblock The calculus of variations.
\newblock 2020.

\bibitem{chertkov2017ensemble}
Michael Chertkov and Vladimir Chernyak.
\newblock Ensemble of thermostatically controlled loads: Statistical physics
  approach.
\newblock {\em Scientific reports}, 7(1):1--9, 2017.

\bibitem{chertkov2018ensemble}
Michael Chertkov, Vladimir~Y Chernyak, and Deepjyoti Deka.
\newblock Ensemble control of cycling energy loads: Markov decision approach.
\newblock In {\em Energy Markets and Responsive Grids}, pages 363--382.
  Springer, 2018.

\bibitem{Dan-KLD-CM}
Y~Dan.
\newblock Bayesian inference for {Gaussian} models: Inverse problems and
  evolution equations.
\newblock 2020.

\bibitem{Stuart-OM}
M.~Dashti, K.~J.~H. Law, A.~M. Stuart, and J.~Voss.
\newblock M{AP} estimators and their consistency in {B}ayesian nonparametric
  inverse problems.
\newblock {\em Inverse Problems}, 29(9):095017, 27, 2013.

\bibitem{dashti2013map}
Masoumeh Dashti, Kody~JH Law, Andrew~M Stuart, and Jochen Voss.
\newblock {MAP} estimators and their consistency in {Bayesian} nonparametric
  inverse problems.
\newblock {\em Inverse Problems}, 29(9):095017, 2013.

\bibitem{Durr}
Detlef D{\"u}rr and Alexander Bach.
\newblock The {Onsager-Machlup} function as {Lagrangian} for the most probable
  path of a diffusion process.
\newblock {\em {Communications in Mathematical Physics}}, 60(2):153--170, 1978.

\bibitem{follmer1985entropy}
Hans F{\"o}llmer.
\newblock An entropy approach to the time reversal of diffusion processes.
\newblock In {\em Stochastic Differential Systems Filtering and Control}, pages
  156--163. Springer, 1985.

\bibitem{follmer1986time}
Hans F{\"o}llmer.
\newblock Time reversal on {W}iener space.
\newblock In {\em Stochastic Processes—Mathematics and Physics}, pages
  119--129. Springer, 1986.

\bibitem{hairer2009introduction}
Martin Hairer.
\newblock An introduction to stochastic {PDEs}.
\newblock {\em arXiv preprint arXiv:0907.4178}, 2009.

\bibitem{Joyce2011}
James~M. Joyce.
\newblock {\em Kullback-Leibler Divergence}, pages 720--722.
\newblock Springer Berlin Heidelberg, Berlin, Heidelberg, 2011.

\bibitem{lehec2013representation}
Joseph Lehec.
\newblock Representation formula for the entropy and functional inequalities.
\newblock In {\em Annales de l'IHP Probabilit{\'e}s et statistiques},
  volume~49, pages 885--899, 2013.

\bibitem{metivier2020mean}
David M{\'e}tivier and Michael Chertkov.
\newblock Mean-field control for efficient mixing of energy loads.
\newblock {\em Physical Review E}, 101(2):022115, 2020.

\bibitem{NuaBook}
David Nualart.
\newblock {\em The {M}alliavin calculus and related topics}.
\newblock Probability and its Applications (New York). Springer-Verlag, Berlin,
  second edition, 2006.

\bibitem{ProtterBook}
Philip Protter.
\newblock {\em Stochastic integration and differential equations}, volume~21 of
  {\em Applications of Mathematics (New York)}.
\newblock Springer-Verlag, Berlin, 1990.
\newblock A new approach.

\bibitem{Revuz}
Daniel Revuz and Marc Yor.
\newblock {\em Continuous martingales and {B}rownian motion}, volume 293 of
  {\em Grundlehren der Mathematischen Wissenschaften [Fundamental Principles of
  Mathematical Sciences]}.
\newblock Springer-Verlag, Berlin, third edition, 1999.

\bibitem{tang2017comparison}
Xun Tang, Jianli Zhang, Michael~A Bevan, and Martha~A Grover.
\newblock A comparison of open-loop and closed-loop strategies in colloidal
  self-assembly.
\newblock {\em Journal of Process Control}, 60:141--151, 2017.

\bibitem{tao1997convex}
Pham~Dinh Tao and Le~Thi~Hoai An.
\newblock Convex analysis approach to {DC} programming: theory, algorithms and
  applications.
\newblock {\em Acta mathematica vietnamica}, 22(1):289--355, 1997.

\bibitem{tao2005dc}
Pham~Dinh Tao et~al.
\newblock The {DC} (difference of convex functions) programming and {DCA}
  revisited with {DC} models of real world nonconvex optimization problems.
\newblock {\em Annals of Operations Research}, 133(1-4):23--46, 2005.

\bibitem{Todorova}
Emanuel Todorov.
\newblock Linearly-solvable {Markov} decision problems.
\newblock In {\em {Advances in Neural Information Processing Systems}}, pages
  1369--1376, 2007.

\bibitem{Todorovb}
Emanuel Todorov.
\newblock Efficient computation of optimal actions.
\newblock {\em {Proceedings of the National Academy of Sciences}},
  106(28):11478--11483, 2009.

\bibitem{todorov2011finding}
Emanuel Todorov.
\newblock Finding the most likely trajectories of optimally-controlled
  stochastic systems.
\newblock {\em IFAC Proceedings Volumes}, 44(1):4728--4734, 2011.

\bibitem{whitelam2020learning}
Stephen Whitelam and Isaac Tamblyn.
\newblock Learning to grow: Control of material self-assembly using
  evolutionary reinforcement learning.
\newblock {\em Physical Review E}, 101(5):052604, 2020.

\end{thebibliography}

\end{document}